\RequirePackage{fix-cm}
\documentclass[smallextended]{svjour3}       
\smartqed  
\usepackage{graphicx}
\usepackage{tikz}
\newtheorem{cor}{Corollary}
\usepackage{enumerate}
\usepackage{url}
\usepackage[mathscr]{euscript}
\begin{document}

\title{Properties of stepwise irregular graphs}


\author{Somnath Bera         \and
        Prithwineel Paul 
}


\institute{Somnath Bera \at
              School of Automation, Huazhong University of Science and technology, \\ Wuhan 430074, Hubei, China.
              \email{somnathbera89@gmail.com}           
           \and
           Prithwineel Paul \at
              Electronics and Communication Sciences Unit, Indian Statistical Institute, \\ Kolkata 700108, India.
              \email{prithwineelpaul@gmail.com}
}

\date{Received: date / Accepted: date}

\maketitle

\begin{abstract}
Stepwise irregular (SI) graphs were introduced by Ivan Gutman recently in 2018 
and
in these graphs the difference between the degrees of any two adjacent vertices is exactly one. In this work, we show 
the existence of connected bicyclic SI graphs of order $5$ and $9$ which has been claimed to be non-existent by Gutman. We also show
the existence of 
SI graphs of different order and cyclomatic numbers when there exist a SI graph with a vertex of degree $1$ or $2$. We give a characterization on the order of the connected tricyclic SI graphs.  
At the end, we investigate some properties of the SI graphs under elementary graph operations along with the 
lower and upper bound of the irregularity of SI graphs.
\end{abstract}

\section{Introduction}



Let 
$G= G(V,E)$ be a connected simple graph with $|V(G)|=n$ 
( $n$ number of vertices) and  
$|E(G)|=m$ ($m$ number of edges).
The order of a graph $G$ is the number of vertices in $G$.
The neighbourhood of a vertex $v$ in $G$ is denoted by $nbd(v)$ and it contains 
all the  vertices adjacent to $v$. 
Its cardinality is called the degree of the vertex $v$. 
The degree of a vertex $v$ in $G$ is often represented by
$d_{G}(v)$. The maximum and minimum degree of the vertices in $G$ is denoted by $\Delta(G)$ and $\delta(G)$ respectively. 
A path of length $(n-1)$ with $n$ vertices is denoted by $P_{n}$.

A graph is regular if all its vertices are of same degree, otherwise it is irregular. To observe the irregularity of a given graph has been an important direction of research.   There are several articles~\cite{dimitovmaximalirregularity,liuhighlyirregular,Gutmanextremalgraph,henningbipartiteirregular,Zofiadegreesequence} in literature to investigate the irregularity measures of a graph.
Also, irregularity of graphs has 
widely been used for analyzing topological  structures of  deterministic and random networks occurring in chemistry, bio-informatics and social networks~\cite{Reginocentrality,estradabiomolecular}.

If $u$ and $v$ are two adjacent vertices in $G$, then the edge between $u$ and $v$ is denoted by $uv$.  
The imbalance of an edge $uv \in E(G)$, defined by Albertson~\cite{Albertson}, is $|d_G(u)-d_G(v)|$ and the irregularity of $G$, which is the most popular and most frequently used graph index, is 
$irr(G)= \sum\limits_{uv \in E(G)} |d_G(u)-d_G(v)|$. 


Among many degree based topological graph indices, two popular indices are 
the first Zagreb index $ M_1(G)$ and the second Zagreb index~\cite{Gutmansecondzagrebindices,dosliczagrebindices,fath-tabarzagrebindices,Gutmanfirstzagrebindices} $M_2(G)$ where

$$ M_1(G)=\sum\limits_{v \in V(G)} d_G(v)^2 \mbox{ and }  M_2(G)=\sum\limits_{uv \in E(G)} d_G(u)d_G(v). $$

There are another multiplicative  versions of Zagreb indices such as 
$$ \Pi_{1}(G)= \prod\limits_{v \in V(G)}d_G(v)^2 \mbox{ and } \Pi_2(G)= \prod\limits_{uv \in E(G)} d_G(u)d_G(v).  $$

The cyclomatic number of a connected graph $G$ with order $n$ and $m$ edges is $\gamma (G)= m-n+1$. The graphs with cyclomatic number $\gamma = 0, 1, 2, 3, 4$ are called tree, unicyclic, bicyclic, tricyclic and tetracyclic graphs respectively.

In this work, we investigate the properties of the stepwise irregular graphs introduced by Ivan Gutman~\cite{SIgraph}. We show the existence of stepwise irregular graphs of order $5$ and $9$. The existence of different stepwise irregular graphs and cyclomatic numbers also have been shown when the precondition of existence of stepwise irregular graph with at least one vertex of degree $1$ or $2$ is satisfied. Moreover, a characterization on the order of the vertices of a tricyclic graph has been provided using the previous results. We further have shown that every integer in the interval $[\delta(G), \Delta(G)]$ must be 
degree of some vertex in $G.$ 
Next we  show that the stepwise irregular graphs are not closed under edge deletion, vertex deletion and complementation. The subdivision, line and total graph of a SI graph is also not SI but the conditions under which these graphs becomes SI have been investigated. The paper ends with the lower and upper bound of the irregularity of stepwise irregular graphs. 

The paper is arranged as follows: At first we recall existing results on SI graphs in Section~\ref{Basic results}. Section~\ref{Main Results} investigates several properties of SI graphs and its existence.  At the end, we study the properties of the SI graphs under elementary graph operations, 
along with the bounds of the irregularity of the SI graphs in section \ref{SI graph under elementary graph operations}.

\section{Basic results}\label{Basic results}
Followings are some important results on stepwise irregular graphs:
\begin{lemma} \cite{SIgraph}
	The number of edges of a stepwise irregular graph is even.
\end{lemma}
\begin{lemma} \cite{SIgraph}
	Stepwise irregular graphs are bipartite.
\end{lemma}

\begin{lemma}\cite{SIgraph} \label{adding 4 vertices having same cyclomatic number}
	Let $G_0$ be an SI graph of order $n_0$ and cyclomatic number $\gamma $, possessing a vertex of degree $1$. Then for all $k = 1, 2, . . . $,
	there exist SI graphs of order $n_0 + 4k$ and cyclomatic number $\gamma $.
\end{lemma}
\begin{theorem}\cite{SIgraph}\label{order of bicyclic graph}
	The order of a stepwise irregular bicyclic graph (graph with $\gamma = 2$) is an odd integer. There exist stepwise irregular	bicyclic graphs whose order is any positive odd integer, except $1, 3, 5, 7, 9$, and $11$.
\end{theorem}
\begin{theorem} \cite{SIgraph} \label{existence of SI graph}
	There exist connected stepwise irregular graphs of any order, except 1, 2, 4, 5, and 6.
\end{theorem}
\section{Main Results}\label{Main Results}
In this section, at first we show that first and second Zagreb index and multiplicative versions of Zagreb indices for stepwise irregular graph are even integers. Next we show that there exist connected bicyclic stepwise irregular graph of order 5 and 9 which is a contradiction to Theorem 
\ref{order of bicyclic graph} and \ref{existence of SI graph}. 

\begin{theorem}
	If $G$ be a SI graph then the graph indices $M_1(G), M_2(G), \Pi_{1}(G)$ and $ \Pi_2(G)$ are even integer. 
\end{theorem}
\begin{proof}
	We see that 	$$ M_1(G)=\sum\limits_{v \in V(G)} d_G(v)^2= \bigg(\sum\limits_{v \in V(G)} d_G(v) \bigg)^2 - 2\sum\limits_{u \ne v  \in V(G)} d_G(u)d_G(v).$$
	
	Now since the sum of degrees of the vertices in $G$ is even number, $ M_1(G)$ is even. 
	
	Since $G$ is a SI graph, for any edge $uv \in E(G)$, $d_G(u)=d_G(v)\pm 1$ and therefore  if $d_G(u)$ is even, then $d_G(v)$ is odd and vice versa. This implies that for any edge $uv \in E(G)$, $d_G(u)d_G(v)$ is even. Thus $M_2(G)= \sum\limits_{uv \in E(G)} d_G(u)d_G(v)$ is even. 
	
	On the other hand,  the other two multiplicative Zagreb indices $\Pi_{1}(G)$ and $ \Pi_2(G)$ must be even as at least one of the degree of a vertex  must be even in a  graph. 
	
\end{proof}
\subsection{On existence of SI graph}

In \cite{SIgraph}, it has been claimed that there do not exist connected bicyclic stepwise irregular graphs of order 5 and 9. In the next result, we show that the claim is in fact not true by constructing examples of connected stepwise irregular graphs of order $5$ and $9$.
\begin{theorem}
	There exist  connected stepwise irregular bicyclic (with cyclomatic number $\gamma =2$) graphs of order $5$ and $9$. 
\end{theorem}
\begin{proof}
	The 
	two graphs shown in Figure~\ref{bicyclic graph of order 5 and 9} proves the above claim.
	\begin{figure}[h]
		\begin{center}
			
			\begin{tikzpicture}
			
			vertices
			\draw[fill=black] (2,0) circle (3pt);
			\draw[fill=black] (0,1.5) circle (3pt);
			\draw[fill=black] (4,1.5) circle (3pt);
			\draw[fill=black] (2,3) circle (3pt);
			\draw[fill=black] (2,1.5) circle (3pt);	
			
			\draw[fill=black] (6,1.5) circle (3pt);
			\draw[fill=black] (6.5,0) circle (3pt);
			\draw[fill=black] (7,1.5) circle (3pt);
			\draw[fill=black] (6.5,3) circle (3pt);
			\draw[fill=black] (7.5,0.5) circle (3pt);
			\draw[fill=black] (7.5,2.5) circle (3pt);
			\draw[fill=black] (8,1.5) circle (3pt);
			\draw[fill=black] (9,1.5) circle (3pt);
			\draw[fill=black] (10,1.5) circle (3pt);

			edges
			\draw[thick] (0,1.5)  --(2,0) -- (2,1.5) --(2,3)  --(0,1.5) (2,3)  --(4,1.5) --(2,0) (6,1.5) --(6.5,0) --(7,1.5) --(6.5,3) --(6,1.5) (6.5,3) --(7.5,2.5) --(8,1.5) --(7.5,0.5) --(6.5,0) (8,1.5) --(9,1.5) --(10,1.5);
			\end{tikzpicture}
			\caption{Bicyclic SI graph of order $5$ and $9$}
			\label{bicyclic graph of order 5 and 9}
		\end{center}
	\end{figure}
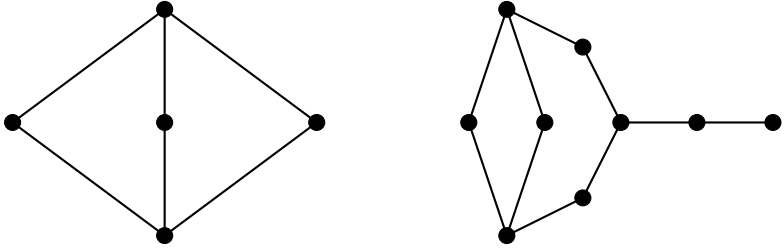	
\end{proof}
Using the above theorem we refine 
the statements of 
Theorem~\ref{order of bicyclic graph} and~\ref{existence of SI graph} as follows: 
\begin{theorem} \label{order of bicyclic graph refined statement}
	The order of a stepwise irregular bicyclic graph (graph with $\gamma = 2$) is an odd integer. There exist stepwise irregular	bicyclic graphs whose order is any positive odd integer, except $1, 3, 7$, and $11$.
\end{theorem}

\begin{theorem}\label{existence of SI graph refined statement}
	There exist connected stepwise irregular graphs of any order, except $1, 2, 4$, and $6$.
\end{theorem}

In the next result, we prove that existence of a SI graph of order $n$ with cyclomatic number $\gamma$ along with a vertex of degree one implies the existence of SI graphs of order $n + 4k + i$ with cyclomatic number $\gamma + i$, $0 \le i \le 2$. This result is a different kind of generalization of the SI graphs same in the line of Lemma \ref{adding 4 vertices having same cyclomatic number}. To prove the result, we need following lemmas.
\begin{lemma} \label{graph operation by adding 5 vertices for higher cyclomatic number}
	Let $G_0$ be a stepwise irregular graph whose vertex u is of degree $1$, cf. Figure~\ref{graph operation by adding 5 vertices}. The stepwise irregular graph $G_1$ can be constructed by adding $5$ new vertices as shown in  cf. Figure~\ref{graph operation by adding 5 vertices}.  
	Moreover, if $G_0$  has cyclomatic number $\gamma$, then $G_1$ has cyclomatic number $\gamma +1$. 
\end{lemma}
\begin{proof}
	From the Figure~\ref{graph operation by adding 5 vertices} itself,  the proof follows. 
	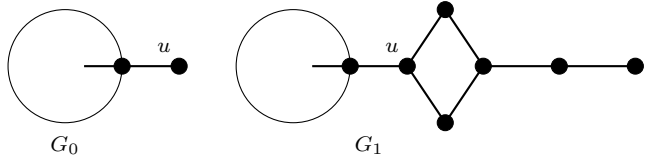
\begin{figure}[h]
		\begin{center}

			\begin{tikzpicture}
			
			vertices
			\draw[fill=black] (2.5,.75) circle (3pt);
			\draw[fill=black] (1.75,.75) circle (3pt);
			\draw (1,0.75) circle (.75cm);
			
			\draw[fill=black] (4.75,.75) circle (3pt);
			\draw[fill=black] (5.5,.75) circle (3pt);
			\draw[fill=black] (8.5,0.75) circle (3pt);
			\draw[fill=black] (7.5,0.75) circle (3pt);
			\draw[fill=black] (6.5,0.75) circle (3pt);
			\draw[fill=black] (6,1.5) circle (3pt);
			\draw[fill=black] (6,0) circle (3pt);
			\draw (4,0.75) circle (.75cm);
			
			vertex labels
			\node at (1,-.3) {$G_0$};
			\node at (5,-.3) {$G_1$};
			\node at (2.3,1) {$u$};
			\node at (5.3,1) {$u$};
			edges
			\draw[thick] (1.25,.75) --(1.75,.75)  --(2.5,.75) (4.25,.75) --(4.75,.75) --(5.5,.75) --(6,1.5) --(6.5,0.75) --(7.5,0.75) --(8.5,0.75) (5.5,.75) --(6,0) --(6.5,0.75);
			\end{tikzpicture}
			\caption{The graph considered in Lemma~\ref{graph operation by adding 5 vertices for higher cyclomatic number}}
			\label{graph operation by adding 5 vertices}
		\end{center}
	\end{figure}
	
\end{proof}
\begin{lemma} \label{graph operation by adding 6 vertices for higher cyclomatic number}
	Let $G_0$ be a stepwise irregular graph whose vertex u is of degree $1$, cf. Figure~\ref{graph operation by adding 6 vertices}. A new stepwise irregular graph $G_1$ can be constructed by adding $6$ vertices as shown in  cf. Figure~\ref{graph operation by adding 6 vertices}.  
	The graph $G_1$ has cyclomatic number $\gamma +2$, if $G_0$  has cyclomatic number $\gamma$. 
\end{lemma}
\begin{proof}
	We construct the graph $G_{1}$ from $G_0$ in the following manner as given in Figure~\ref{graph operation by adding 6 vertices}, which proves the above claim. 
	
	\begin{figure}[h]
		\begin{center}

			\begin{tikzpicture}
			
			vertices
			\draw[fill=black] (2.5,.75) circle (3pt);
			\draw[fill=black] (1.75,.75) circle (3pt);
			\draw (1,0.75) circle (.75cm);
			
			\draw[fill=black] (4.75,.75) circle (3pt);
			\draw[fill=black] (5.5,.75) circle (3pt);
			\draw[fill=black] (6,0) circle (3pt);
			\draw[fill=black] (6,1.5) circle (3pt);
			\draw[fill=black] (6.5,.75) circle (3pt);
			\draw[fill=black] (7,0) circle (3pt);
			\draw[fill=black] (7,1.5) circle (3pt);
			\draw[fill=black] (7.5,.75) circle (3pt);
			\draw (4,0.75) circle (.75cm);
			
			vertex labels
			\node at (1,-.3) {$G_0$};
			\node at (5,-.3) {$G_1$};
			\node at (2.3,1) {$u$};
			\node at (5.3,1) {$u$};
			edges
			\draw[thick] (1.25,.75) --(1.75,.75)  --(2.5,.75) (4.25,.75) --(4.75,.75) --(5.5,.75) --(6,1.5) --(7,1.5) --(7.5,0.75) --(7,0) --(6,0) --(5.5,.75) (7,0) --(6.5,0.75) --(7,1.5);
			\end{tikzpicture}
			\caption{The graph considered in Lemma~\ref{graph operation by adding 6 vertices for higher cyclomatic number}}
			\label{graph operation by adding 6 vertices}
		\end{center}
	\end{figure}
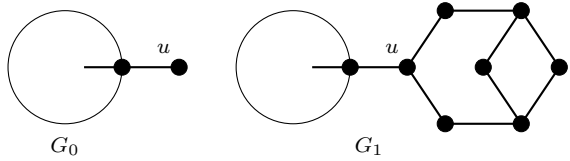
\end{proof}

Using  Lemma~\ref{graph operation by adding 5 vertices for higher cyclomatic number} and~\ref{graph operation by adding 6 vertices for higher cyclomatic number}, we prove the next theorem.
\begin{theorem}
	Let $G$ be a SI graph of order $n$ and cyclomatic number $\gamma$, possessing a vertex of degree $1$. Then for all $k=1,2,...,$ there exist SI graphs of order $n+4k+i$ and cyclomatic number $\gamma +i$, for $i=0,1,2$.
\end{theorem}
\begin{proof}
	\begin{enumerate}
		\item $i=0$ is same as Lemma~\ref{adding 4 vertices having same cyclomatic number}. 
		
		\item {\bf $i=1$:} for $k=1$ by application of Lemma~\ref{graph operation by adding 5 vertices for higher cyclomatic number} once we get the required graph. For $k \ge 2$, by  application of Lemma~\ref{graph operation by adding 5 vertices for higher cyclomatic number} once and then applying Lemma~\ref{adding 4 vertices having same cyclomatic number} $(k-1)$ times, we get the required graph.
		
		\item {\bf $i=2$:} for $k=1$, by one time application of Lemma~\ref{graph operation by adding 6 vertices for higher cyclomatic number} and for $k=2$, by two times application of Lemma~\ref{graph operation by adding 5 vertices for higher cyclomatic number},  we get the required graph respectively.  
		
		For $k = 3$, we can get the required graph by two times application of Lemma \ref{graph operation by adding 5 vertices for higher cyclomatic number} and one time application of Lemma~\ref{adding 4 vertices having same cyclomatic number} respectively.
		
		For $k \ge 4$, by two times application of Lemma~\ref{graph operation by adding 5 vertices for higher cyclomatic number} and $(k-2)$ times Lemma~\ref{adding 4 vertices having same cyclomatic number} we get the required graph.
		
	\end{enumerate} 
\end{proof}
\begin{cor}
	Let $G$ be a SI graph of order $n$ and cyclomatic number $\gamma$, possessing a vertex of degree $1$. Then for all $k=2,3...,$ there exist SI graphs of order $n+4k+3$ and cyclomatic number $\gamma +3$.
\end{cor}
\begin{proof}
	For $k=2$, first applying Lemma~\ref{graph operation by adding 5 vertices for higher cyclomatic number} and then Lemma~\ref{graph operation by adding 6 vertices for higher cyclomatic number} we get the required graph.
	
	For $k = 3$, the required graph can be constructed by application of the Lemma \ref{graph operation by adding 5 vertices for higher cyclomatic number}, $3$ times. 
	
	For $k \ge 4$, first applying lemma~\ref{graph operation by adding 5 vertices for higher cyclomatic number}, $3$ times and then Lemma~\ref{adding 4 vertices having same cyclomatic number}, $(k-3)$ times, we get the required graph.
\end{proof}

In the above results, we proved that the existence of a SI graph with a pendant vertex imply the existence of SI graphs of different order and cyclomatic number. In the following results we show that the existence of a SI graphs with a vertex of degree $2$ having neighbours of degree $3$ proves the existence of SI graphs of order $n + 7k$ and cyclomatic number $\gamma + k$. 

To prove the above result, we will use the following lemma.
\begin{lemma} \label{graph operation for higher cyclomatic number}
	Let $G_0$ be a stepwise irregular graph whose vertex u is of degree 2 having neighbors each of which are of degree $3$, cf. Figure~\ref{graph operation}. The SI graph $G_1$ can be constructed by adding $7$ new vertices as shown in  cf. Figure~\ref{graph operation}.  
	Moreover, if $G_0$  has cyclomatic number $\gamma$, then $G_1$ has cyclomatic number $\gamma +1$. 
\end{lemma}
\begin{figure}[h]
	\begin{center}

		\begin{tikzpicture}
		
		vertices
		\draw[fill=black] (1,.3) circle (3pt);
		\draw[fill=black] (2,.75) circle (3pt);
		\draw[fill=black] (1,1.1) circle (3pt);
		\draw (1,0.75) circle (.75cm);
		
		\draw[fill=black] (4,.3) circle (3pt);
		\draw[fill=black] (5,.75) circle (3pt);
		\draw[fill=black] (4,1.1) circle (3pt);
		\draw[fill=black] (5.5,1.5) circle (3pt);
		\draw[fill=black] (5.5,0) circle (3pt);
		\draw[fill=black] (6,0.75) circle (3pt);
		\draw[fill=black] (6,1.5) circle (3pt);
		\draw[fill=black] (6.5,1.5) circle (3pt);
		\draw[fill=black] (6.5,0) circle (3pt);
		\draw[fill=black] (6,0) circle (3pt);
		\draw (4,0.75) circle (.75cm);
		
		vertex labels
		\node at (1,-.3) {$G_0$};
		\node at (5,-.3) {$G_1$};
		\node at (2,1.1) {$u$};
		\node at (4.9,1.1) {$u$};
		edges
		\draw[thick] (1,.3)  --(2,.75) -- (1,1.1) (.9, .8) --(1,1.1) --(.8,1.3) (.7,.3) --(1,.3) --(.9, .6) (4,.3)  --(5,.75) -- (4,1.1) (3.9, .8) --(4,1.1) --(3.8,1.3) (3.7,.3) --(4,.3) --(3.9,.6) (5,.75) --(5.5,1.5) --(6,1.5) --(6.5,1.5) (5.5,1.5) --(6,0.75) --(5.5,0) (5,.75) --(5.5,0) --(6,0) --(6.5,0);
		\end{tikzpicture}
		\caption{The graph considered in Lemma~\ref{graph operation for higher cyclomatic number}}
		\label{graph operation}
	\end{center}
\end{figure}
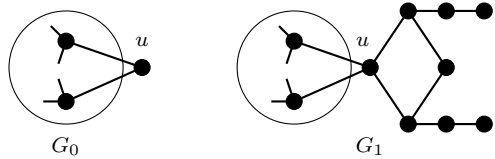


Now by repeated application of Lemma~\ref{graph operation for higher cyclomatic number} we can claim the following  	
\begin{theorem}
	Let $G_0$ be a SI graph of order $n$ and cyclomatic number $\gamma$ , possessing a vertex of degree $2$ having neighbors each of which are of degree $3$. Then for all $ k = 1, 2,...$, there exist SI graphs of order $n + 7k$ and cyclomatic number $\gamma +k$.
\end{theorem}
The next theorem can be shown as a consequence of
Lemma \ref{graph operation for higher cyclomatic number} and Lemma \ref{adding 4 vertices having same cyclomatic number}. In fact it can be  proved by one time application of Lemma~\ref{graph operation for higher cyclomatic number} and $(k-1)$ times application of  Lemma~\ref{adding 4 vertices having same cyclomatic number} to the graph $G_{0}$. 

\begin{theorem}
	Let $G_0$ be a SI graph of order $n$ and cyclomatic number $\gamma$ , possessing a vertex of degree $2$ having neighbors each of which are of degree $3$. Then for all $ k = 1, 2,...$, there exist SI graphs of order $n + 4k+3$ and cyclomatic number $\gamma +1$.
\end{theorem}

In \cite{SIgraph}, a characterization on the order of the connected bicyclic SI graphs was given by Gutman. In this work, we give a characterization on the order of the connected tricyclic SI graphs.  
\begin{theorem}
	The order of a SI tricyclic graph is an even integer. Moreover,  there exist tricyclic SI graphs whose order is any positive integer except $2,4,6$ and $8$. 
\end{theorem}
\begin{proof}
	By Theorem~\ref{order of bicyclic graph refined statement}, it can be shown that there exist bicyclic SI graph order $13+2k$, $k \ge 0$. Now by applying Lemma~\ref{graph operation by adding 5 vertices for higher cyclomatic number}, it can further be shown that there exist tricyclic SI graph of order $13+2k+5$, i.e., $18+2k$, $k \ge 0$.
	
	Also, there exist bicyclic SI graph of order $5$ and $9$. Hence, the Lemma \ref{graph operation for higher cyclomatic number} proves the existence of tricyclic SI graph of order $12$ and $16$. 
	
	Gutman has already proved the existence of  tricyclic SI graphs of order $10$ and $14$ in ~\cite{SIgraph}.
	
	On the other hand, such tricyclic SI graph of order $2,4,6$ and $8$ do not exist. 
\end{proof}

In the next result, we show that each integer between the minimum and maximum degree of a SI graph, must be a degree of some vertex of the SI graph.  

\begin{theorem}
	Let $G$ be a connected SI graph with maximum degree $\Delta(G)$ and minimum degree $\delta(G)$. Then every number $k$,  $\delta(G) \le k \le \Delta(G)$ belongs to the degree set  $d(G)$
	of $G$.
\end{theorem}
\begin{proof}
	Suppose $G$ be a connected SI graph with maximum degree $M$ and minimum degree $m$. Clearly $m$ and $M$ belong to $d(G)$. 
	
	Now if possible let there exists $p$, $m < p < M$ such that $p \notin d(G)$. 
	
	If $p+1 \in d(G)$, set $S$ be the set of all vertices whose degrees are greater than $p$. Then none of these vertices are connected to any vertices of degree less than $p$. This implies that $G$ is disconnected - a contradiction. 
	
	If $p+1 \notin d(G)$, by a finite number of steps we can always find a $k$ such that $p+k \le M$ and $p+k \in d(G)$. Using the same argument above, it can be shown that $G$ is disconnected - a contradiction. 
	
	Hence our assumption is wrong and the theorem is proved. 
\end{proof}

\textbf{Note:} There exist SI graphs which have only vertices of degree $\Delta(G)$ and $\delta(G)$.  The complete bipartite  graph $G = k_{m,m+1}, m \geq 1$ is an example of such kind of graphs.

\section{SI graph under elementary graph operations}
\label{SI graph under elementary graph operations}

In this section we investigate the SI property of graphs under several graph operations such as vertex (edge resp.) deleted subgraph, compliment graph, line and total graphs etc.

For a given graph $G$ and a vertex $v \in V(G)$, the 
vertex  deleted 
subgraph  $G-v$ is obtained by deleting the vertex $v$ and its adjacent edges from $G$. Similarly, for a given graph $G$ and an edge $e \in E(G)$, the 
edge deleted 
subgraph $G-e$ is obtained by deleting the edge $e$ from $G$. 

After deleting either a vertex or an edge from a given connected graph we may get disconnected graph. A disconnected graph $G$ is said to be SI if each component of $G$ is SI. 

We show in the next result that edge deletion does not preserve the property of stepwise irregularity.
\begin{theorem}
	Let $G$ be a SI graph. Then for any  edge $e$  of $G$, $G-e$ is  no longer SI. 
\end{theorem}
\begin{proof}
	Let $e=xy$ be an edge of a SI graph $G$ and $H=G - e$.  Also we have the order of $G$ at least $3$. 
	
	Consider two adjacent vertices $u$ and $v$ in $H$. Then these $u$ and $v$ are also in $G$.  Therefore we have $d_G(u)= d_G(v)\pm 1$. 
	
	There may arise two cases. 
	\begin{enumerate}
		\item Neither  $u$ nor $v$ is $x$ or $y$. Then $|d_H(u)- d_H(v)|=|d_G(u)-d_G(v)| =1 $.
		
		\item Either  $u$ or $v$ (but not both, since we remove the edge $e=xy$ ) is $x$ or $y$. Without loss of generality let $u=x$ and $v \in V(G) - \lbrace x, y \rbrace$. Then  $|d_H(u)- d_H(v)|=|d_G(u)- 1-d_G(v)| = |d_G(v)\pm 1 - 1-d_G(v)| = 0 \mbox{ or } 2$. 
	\end{enumerate}
	Since the SI graph $G$ has at least $2$ edges, there must exist an edge which falls under the second case. 
	
	By the second case it follows that $H$ is not SI. 
\end{proof}
Now in a similar manner  we show that, vertex deletion does not preserve the property of stepwise irregularity.
\begin{theorem}
	Let $G$ be a SI graph. Then for any vertex  $v$ of $G$,  $G-v$ is  not SI.
\end{theorem}
\begin{proof}
	Suppose $G$ be a SI graph and $v$ be a vertex of $G$. Let $H=G - v$.
	
	Consider two adjacent vertices $x$ and $y$ in $H$. Then these $x$ and $y$ are also in $G$.  Therefore we have $d_G(x)= d_G(y)\pm 1$.  There may arise two cases. 
	\begin{enumerate}
		\item 	If $x, y \in V(G)-nbd(v)$, $|d_H(x)- d_H(y)|=|d_G(x)- d_G(y)|=1$.
		
		\item Either of $x$ or $y$ (but not both) is $v$ ( it is because SI graphs are bipartite and if they are both in $nbd(v)$, $\lbrace v, x, y\rbrace$ will form a triangle ). 
		
		Without loss of generality let $x\in nbd(v)$ and $y \in V(G)- nbd(v)$. Then  $|d_H(x)- d_H(y)|=|d_G(x)- 1-d_G(y)| = |d_G(y)\pm 1 - 1-d_G(y)| = 0 \mbox{ or } 2$. 
	\end{enumerate}
	By the second case it follows that $H$ is not SI. 
	
\end{proof}
Moreover,  in the next result we show that like in the case of vertex and edge deletion, the complementation operation also does not preserve stepwise irregularity. 
\begin{lemma}
	Let $G$ be a connected SI graph, then $\overline{G}$, the compliment graph of $G$ is not SI.
\end{lemma}
\begin{proof}
	Let $G$ be a connected SI graph of order $n$ 
	with maximum degree $M$ and minimum degree $m$. Then there must exist at least two vertices $u, v$ of same degree $k$, for some $k$, $m \le k \le M$,  in $G$, which are not adjacent to each other. So the vertices $u$ and $v$ are adjacent in $\overline{G}$ having same degree $(n-1-k)$. 
	Hence $\overline{G}$ is not SI. 
\end{proof}

\subsection{Subdivision graph of SI graphs}
\label{Subdivision graph of SI graphs}

The subdivision graph~\cite{Harary1969} of a graph $G$ is denoted by $S(G)$ and is obtained from $G$ by inserting a new vertex in each of the edges of $G$.



\begin{theorem}
	If $G$ is a SI graph, then  the subdivision graph $S(G)$ is not SI.
\end{theorem}
\begin{proof}
	Let $G$ be a SI graph and $e=uv$ be an edge of $G$. Then $d_{G}(u)=d_{G}(v)\pm 1$. Now in the subdivision graph $S(G)$, a new vertex $w$ (say) is inserted in $e$. We see that $d(w)= 2$ in $S(G)$. 
	Hence, either $|d_{S(G)}(u)-d_{S(G)}(w)| \ne 1$ or $|d_{S(G)}(v)-d_{S(G)}(w)| \ne 1$. This implies that $S(G)$ is not SI. 
\end{proof}

\begin{example}\label{3 regular SI}
	There exist non-SI graphs whose subdivision graph is SI. 
	A $3$-regular graph is an example of such kind of graph. 
\end{example}
Moreover, we can show that there exist only three types of graph for which the subdivision graph can be a SI graph.
\begin{theorem}
	The subdivision graph of any connected graph other than $K_2$, $K_{1,3}$ and $3$-regular graphs, is not SI. 
\end{theorem}
\begin{proof}
	If  $G= K_2$, $S(G)= K_{1,2}$ which is SI. If $G= K_{1,3}$, it can be directly verified that $S(G)$ is SI. By Example~\ref{3 regular SI}, the subdivision graph of $3$-regular graph is SI. 
	
	Assume that $G$ is a graph other than $K_2$, $K_{1,3}$ and $3$-regular graphs. Then there exists at least one edge $uv \in E(G)$ whose one of adjacent vertices has degree other than $1$ and $3$.
	In fact let the vertex  $u \in V(G)$ be such that $d(u)\ne 1, 3$. Now in the subdivision graph $S(G)$, a new vertex $w$ (say) is inserted onto this edge and the degree of this new vertex is $2$ which is also adjacent to the vertex $u$. Clearly, $|d(u)-d(w)| \ne 1$. Therefore $S(G)$ can not be SI.
\end{proof}


\subsection{Line and total graph of SI graph}

The line graph $L(G)$ of a graph $G$ 
is the graph whose vertices correspond to the edges of $G$ and any two vertices in $L(G)$ are adjacent if and only if the edges corresponding to the vertices are adjacent in $G$ \cite{Harary1969}. 
Moreover, if $e = uv$ is an edge of $G$ then $deg_{ L(G)} (e) = deg_{G} (u)+ deg_{G} (v) - 2$. 

\begin{theorem}\label{line graph not si}
	
	Let $G$ be a SI graph, then $L(G)$, the line graph of $G$, is not SI. 
\end{theorem}
\begin{proof}
	Let $G$ be a SI graph. Then $G$ has at least two edges. Let $e_1=uv$ and $e_2=uw$ be two adjacent edges in $G$.  Then we have, $d_{G}(v)=d_{G}(u)\pm 1$ and $d_{G}(w)=d_{G}(u)\pm 1$.
	
	Since $e_1$ and $e_2$ are adjacent in $G$, they are adjacent in $L(G)$. We also have, $d_{L(G)}(e_1)=d_{G}(u)+d_{G}(v)-2$ and $d_{L(G)}(e_2)=d_{G}(u)+d_{G}(w)-2$. 
	
	Therefore $|d_{L(G)}(e_1)-d_{L(G)}(e_2)|=|d_{G}(v)-d_{G}(w)|= |(d_{G}(u)\pm 1)-(d_{G}(u)\pm 1)|= 0 \mbox{ or } 2$. This implies that $L(G)$ is not SI. 
\end{proof}

\textbf{Note:} There exist non-SI graphs whose line graph can be SI. The graph $P_{4}$ is not an SI graph, but the line graph of $P_{4}$ is $P_{3}$ which is a SI graph. 

In fact, in 
the following result we show that $P_4$ is the only graph whose line graph is a SI graph.
\begin{theorem}
	$G=P_4$ if and only if $L(G)$ is a SI graph.
\end{theorem}
\begin{proof}
	If $G=P_4$, then $L(G)=P_3$ which is SI. 
	
	To prove the converse, suppose $L(G)=H$ (say) is a SI graph. If  $L(G)=P_3$, we are done. 
	If possible let, $H \ne P_3$. Then we claim that $H$ has $K_{1,3}$ as induced subgraph. 
	
	Proof of the claim: Since H is SI, there are two cases:
	\begin{enumerate}
		\item If $H= K_{m,m+1}$, $m \ge 2$, then clearly it has $K_{1,3}$ as induced subgraph.
		\item If $H \ne K_{m,m+1}$, $m \ge 2$, then let the degree set of $H$ be $\lbrace m, m+1, ... , M\}$ with minimum degree $m \ge 1$ and maximum degree $M$. Now consider the induced subgraph by the set of a vertex of degree $k+1$ ($\ge 3$) where $m <k < M$ and its any $3$ adjacent vertices. We see that the induced subgraph forms $K_{1,3}$.  
	\end{enumerate}
	
	Therefore using the fact (\cite{Harary1969}, pp 74--75) that a graph having $K_{1,3}$ as induced subgraph can not be a line graph, we get a contradiction that $H$ is a line graph. 
	
	Hence the only possibility is  $L(G)=P_3$.
	
	This implies that $G=P_4$.
\end{proof}
Using the definition of line graph the following can be shown. 
\begin{theorem}
	Let $G$ be a SI graph. Then the degree of each vertex of the line graph $L(G)$ is an odd integer.
\end{theorem}

The total graph $T (G)$ of a graph $G$ is a graph whose vertex set is $V (G) \cup E(G)$ and two vertices in $T(G)$ are adjacent if and only if the corresponding elements are adjacent or incident in $G$ \cite{Harary1969}. If $u$ is a vertex of $G$, then $deg_{ T (G)} (u) = 2 deg_{G} (u)$. If $e = uv$ is an edge of $G$ then $deg_{ T (G)} (e) = deg_{G} (u) + deg_{G} (v)$.

Further more, it is noted that every edge in a graph $G$ (say) produces at least one triangle in $T(G)$. Therefore, the total graph of any graph is not bipartite. This implies the following lemmas. 


\begin{lemma} \label{total graph not si}
	Let $G$ be a SI graph, then $T(G)$, the total graph of $G$, is not SI. 
\end{lemma}

\begin{lemma}
	Let $G$ be a graph, then the total graph $T(G)$ can never be SI.
\end{lemma}



In~\cite{Albertson}, Albertson claimed that the irregularity of a bipartite graph is maximum if it is a complete bipartite graph. Using this result we give a lower and upper bound for the number of edges of a SI graph of given order and show that both the bounds are tight. 
\begin{theorem}
	Let $G$ be a connected SI graph of order $n$ and $m$ be the number of edges of $G$. Then $m$ is even and $ n-1 \le m \le \frac{n^2-1}{4}$. Moreover, for odd $n$, the lower and upper bound is attained when $G$ is a SI tree graph and $K_{\frac{n-1}{2}, \frac{n+1}{2}}$ respectively. 
\end{theorem}
\begin{proof}
	Let $G$ be a connected SI graph of order $n$ and $m$ edges. Since $G$ is connected, it must have $n-1$ edges. Therefore, $m \ge n-1$. And the equality holds for all SI tree graphs.  
	
	To prove the upper bound, since $G$ is SI, it is bipartite and the irregularity of $G$ is same as the number of edges of $G$. It is known that the irregularity of a bipartite graph is maximum if it is a complete bipartite graph. The maximum irregularity of $G$ can be obtained by partitioning $n$ vertices into two vertex sets in such a way that the number of edges is maximum. Let the partition of $n$ be $n= \frac{n-k}{2} +\frac{n+k}{2}$, for some $ k \ge 1$. The complete bipartite graph $K_{\frac{n-k}{2}, \frac{n+k}{2}}$ has $\frac{n^2-k^2}{4}$ edges.  Thus  $m= irr(G) \le \frac{n^2-k^2}{4} \le \frac{n^2-1}{4}$. 
	
	The equality holds if $n$ is odd and $k=1$.  In this case we get $G=K_{\frac{n-1}{2}, \frac{n+1}{2}}$ which is SI. Therefore for complete bipartite SI graph $K_{\frac{n-1}{2}, \frac{n+1}{2}}$, the equality holds.

\end{proof}




{\bf Note: } By the definition of SI graphs, it is clear that the disjoint union of Si graphs is SI. The join of two graphs $G$ and $H$, denoted by $G+H$, is the graph obtained by $V(G+H)= V(G)\cup V(H)$ and $E(G+H)= E(G)\cup E(H)\cup \lbrace uv : u \in V(G), v \in V(H) \rbrace$. Using the definition of join of two graphs it can be shown that for any two non trivial graphs $G$ and $H$, the join $G+H$ is not bipartite. Hence the join can never be SI graph unless $G= \overline{K}_{m}$, $H= \overline{K}_{m+1}$. 

Product graphs are very useful to construct large graphs from small graphs and has been used to design interconnection networks. Some of products of two graphs  are lexicographic, direct, Cartesian and strong product~\cite{imrich2000product}. Among these products, lexicographic and strong product of two non trivial graphs is not bipartite, therefore not SI. On the other hand, it can be shown that the direct and Cartesian product of two graphs need not be SI. 

\section{Acknowledgements}
The first author S. Bera gratefully acknowledges support by National Natural Science Foundation of China (61320106005 and 61772214) and the Innovation Scientists and Technicians Troop Construction Projects of Henan Province (154200510012).
\bibliographystyle{spmpsci}

\bibliography{mybib}



%
%

\end{document}